\documentclass[12pt]{article}   
\usepackage{graphicx,amssymb,amsmath,amsthm,subfig}
\usepackage{appendix}
\newcommand{\bY}{\mathbf{Y}}
\newcommand{\bN}{\mathbf{N}}
\newcommand{\bZ}{\mathbf{Z}}
\newcommand{\Me}{\mathrm{Me}}
\newcommand{\Mee}{\widehat{\Me}}
\newcommand{\Fe}{\widehat{F}}
\newcommand{\R}{\mathbb{R}}
\newcommand{\N}{\mathbb{N}}
\renewcommand{\P}{\mathrm{P}}
\newcommand{\dd}{\mathrm{d}}

\newtheorem{theorem}{Theorem}[section]
\newtheorem{proposition}{Proposition}[section]
\newtheorem{lemma}{Lemma}[section]

\begin{document}

\title{The median of a jittered Poisson distribution}

\author{Jean-Fran\c{c}ois Coeurjolly and Jo\"elle Rousseau-Tr\'epanier \\
              UQAM, Montr\'eal \\
              \texttt{coeurjolly.jean-francois@uqam.ca}\\
              \texttt{rousseau\_trepanier.joelle@courrier.uqam.ca}
}

\date{}

\maketitle

\begin{abstract}
Let $N_\lambda$ and $U$ be two independent random variables respectively distributed as a Poisson distribution with parameter $\lambda >0$ and a uniform distribution on $(0,1)$. This paper establishes that the median, say $M$, of $N_\lambda+U$ is close to $\lambda +1/3$ and more precisely that $M-\lambda-1/3=o(\lambda^{-1})$ as $\lambda\to \infty$. This result is used to construt a very simple robust estimator of $\lambda$ which is consistent and asymptotically normal. Compared to known robust estimates, this one can still be used with large datasets ($n\simeq 10^9$).
\end{abstract}

\section{Introduction and position of the problem}

The Poisson distribution is commonly used for modeling count data. Let $\bN_\lambda=(N_{1,\lambda},\dots, N_{n,\lambda})$ be a sample of $n\geq 1$ independent and identically distributed random variables distributed as $N_\lambda$ a Poisson distribution with parameter $\lambda >0$. Different strategies exist to make the maximum likelihood estimator of $\lambda$ more robust to outliers. For example, specific M-estimators (such as the modified Tukey's type estimate) for $\lambda$ have been investigated deeply by~\cite{elsaied:fried:16}. The authors also investigate weighted likelihood type estimators and trimmed-mean estimators. In the present paper, we focus our attention on the simplest robust alternative to the maximum likelihood estimator, which is the sample median of $\bN_\lambda$ and actually on a theoretical problem induced by the use of such an estimator. 

To introduce our contribution, let us consider the following standard notation. For a  random variable $Y$, we denote by $F_Y(\cdot)$ its cumulative distribution function (cdf), by $F_Y^{-1}(p)$ its quantile of order $p\in (0,1)$ and by $\Me_Y=F_Y^{-1}(1/2)$ its theoretical median. Based on a sample $\mathbf{Y}=(Y_1,\dots,Y_n)$ of $n$ identically distributed random variables we denote by $\Fe(\cdot;\mathbf Y)$ the empirical cdf, by $\Fe^{-1}(p;\mathbf Y)$ the sample quantile of order $p$ given by
% \begin{equation} \label{def:quantile}
  $\Fe^{-1}(p;\mathbf Y) = \inf \{ x\in \R: p \leq \Fe(x;\bY)\}.$
% \end{equation}
The sample median is simply denoted by $\Mee(\mathbf Y)=\Fe^{-1}(1/2;\mathbf Y)$. Finally, the density of $Y$ when it exists, is denoted by $f_Y$.

 Due to the discrete nature of the Poisson distribution, the limiting distribution of $\Mee(\bN_\lambda)$ does not follow from standard theory, see e.g.~\cite{vandervaart} or~\cite{serfling:09}, since it is required that the model possesses a positive density at the true median. To circumvent this problem, one classical strategy introduced by~\cite{stevens:50}, applied to count data by~\cite{machado:05} and to the estimation of the intensity of a homogeneous spatial point process by~\cite{coeurjolly:17}, consists in artificially imposing smoothness in the problem through jittering: i.e. we add to each count variable $N_{i,\lambda}$ a random variable $U_i\sim \mathcal U((0,1))$. Let $\bZ_\lambda=(Z_{1,\lambda},\dots,Z_{n,\lambda})$, where $Z_{i,\lambda}=N_{i,\lambda}+U_i$ for $i=1,\dots,n$, be the sample of independent random variables distributed as $Z_\lambda=N_\lambda+U$ where $U\sim \mathcal U((0,1))$ is independent of $N_\lambda$. It can be shown that $Z_\lambda$ admits a density almost everywhere, which is given by 
\begin{equation}\label{eq:fZlambda}
  f_{Z_\lambda}(t) = \P(N_\lambda = \lfloor t \rfloor), \quad t\ge 0.
\end{equation}
Standard asymptotic theory (e.g. \cite{serfling:09}) is now valid: as $n\to \infty$
\begin{equation} \label{eq:TCL_Mee}
  \sqrt{n} \left( \Mee(\bZ_\lambda) -\Me_{Z_\lambda}\right) \to N(0,\sigma_\lambda^2)
\end{equation}
in distribution, where $\sigma^{-2}_\lambda=4 f_{Z_\lambda}(\Me_{Z_\lambda})^2$. 

Equation~\eqref{eq:TCL_Mee} is the source of motivation for the present paper since we are clearly invited to understand how far $\Me_{Z_\lambda}$ is from $\lambda$. The study of the median for Poisson and Gamma distributions has a long story, see~\cite{choi:94} and the references therein. We can even go back to an old and outstanding formula by Ramanujan, see \cite[Equation (3)]{choi:94}. Among several results, \cite{choi:94} proves a conjecture proposed by~\cite{chen:rubin:86} which is that for every $\lambda >0$ 
\[
  -\log 2 \leq \Me_{N_\lambda}-\lambda \leq \frac13.
\]
It is worth mentioning that these bounds are optimal, in the sense that there exists at least one value of $\lambda$ for which the lower-bound or upper-bound is reached. \cite{adell:jodra} complete this work and prove that asymptotically as $\lambda \to~\infty$
\[
  \liminf_{\lambda \to \infty} \, \Me_{N_\lambda} -\lambda= -\frac23 \quad \text{ and }\quad \limsup_{\lambda \to \infty} \, \Me_{N_\lambda} -\lambda= \frac13.
\]
Going back to $\Me_{Z_\lambda}$, using these  results one easily deduces that
\[
  -\log(2) \le \Me_{Z_\lambda} - \lambda  \le \frac43.
\]
Such a result is definitely pessimistic since the contribution of this paper is to show that we have the surprising and unexpected following result: $\Me_{Z_\lambda}$ is actually very close to $\lambda+1/3$. Even more, our main result implies that, by denoting $\delta_\lambda= \Me_{Z_\lambda} -\lambda -1/3$
\begin{equation}
  \label{eq:liminfsup}
\liminf_{\lambda \to \infty} \, \lambda\, \delta_\lambda = -\frac{8}{405} 
\quad \text{ and } \quad 
\limsup_{\lambda \to \infty} \lambda \, \delta_\lambda = \frac4{135}.
\end{equation}
The latter results suggests us to propose $\hat \lambda^\mathrm{J} = \Me(\bZ_\lambda)-1/3$ as a new estimator for $\lambda$.

The rest of the paper is organized as follows. Section~\ref{sec:result} presents our main result and provides a sketch of the proof while Section~\ref{sec:sim} illustrates this result.
We investigate statistical properties of $\hat \lambda^\mathrm{J}$ and compare its performances with the maximum likelihood estimator and the Tukey's modified estimator proposed by~\cite{elsaied:fried:16}. Finally, we show that $\hat \lambda^\mathrm{J}$ does not suffer from computational problems and can still be used with very large datasets.
The proof of our main result relies upon simple technical lemmas which are postponed to Appendix.

\section{Main result} \label{sec:result}

We consider the notation introduced in the previous section. Let us first mention that the cumulative distribution function $F_{Z_\lambda}$ is given for any $t\ge 0 $ by
\begin{equation}
  \label{eq:FZlambda}
  F_{Z_\lambda}(t)=\P{(Z_\lambda\leq t)} = \sum_{k=0}^{\lfloor t \rfloor} \P(N_\lambda = k) 
  + (t-\lfloor t \rfloor) \, \P (N_\lambda = \lfloor t \rfloor),
\end{equation}
whereby it can be checked that $Z_\lambda$ indeed admits a density almost everywhere and that this density is given by~\eqref{eq:fZlambda}. Our main result is based on the empirical finding depicted in Figure~\ref{fig1}. Figure~\ref{fig1}~(a) illustrates that for any $\lambda>0$, $\Me_{N_\lambda}-\lambda \in [-\log(2),1/3]$, while Figure~\ref{fig1}~(b)-(c) illustrate that indeed $\Me_{Z_\lambda} \approx \lambda+1/3$. Note that to evaluate $\Me_{Z_\lambda}$ we use root-finding algorithm for the function $|F_{Z_\lambda}(\cdot)-1/2|$. We now present our main result.

\begin{figure}[htbp]
\centering\begin{tabular}{cc}
\subfloat[{$\Me_{N_\lambda}-\lambda$}]{\includegraphics[scale=.45]{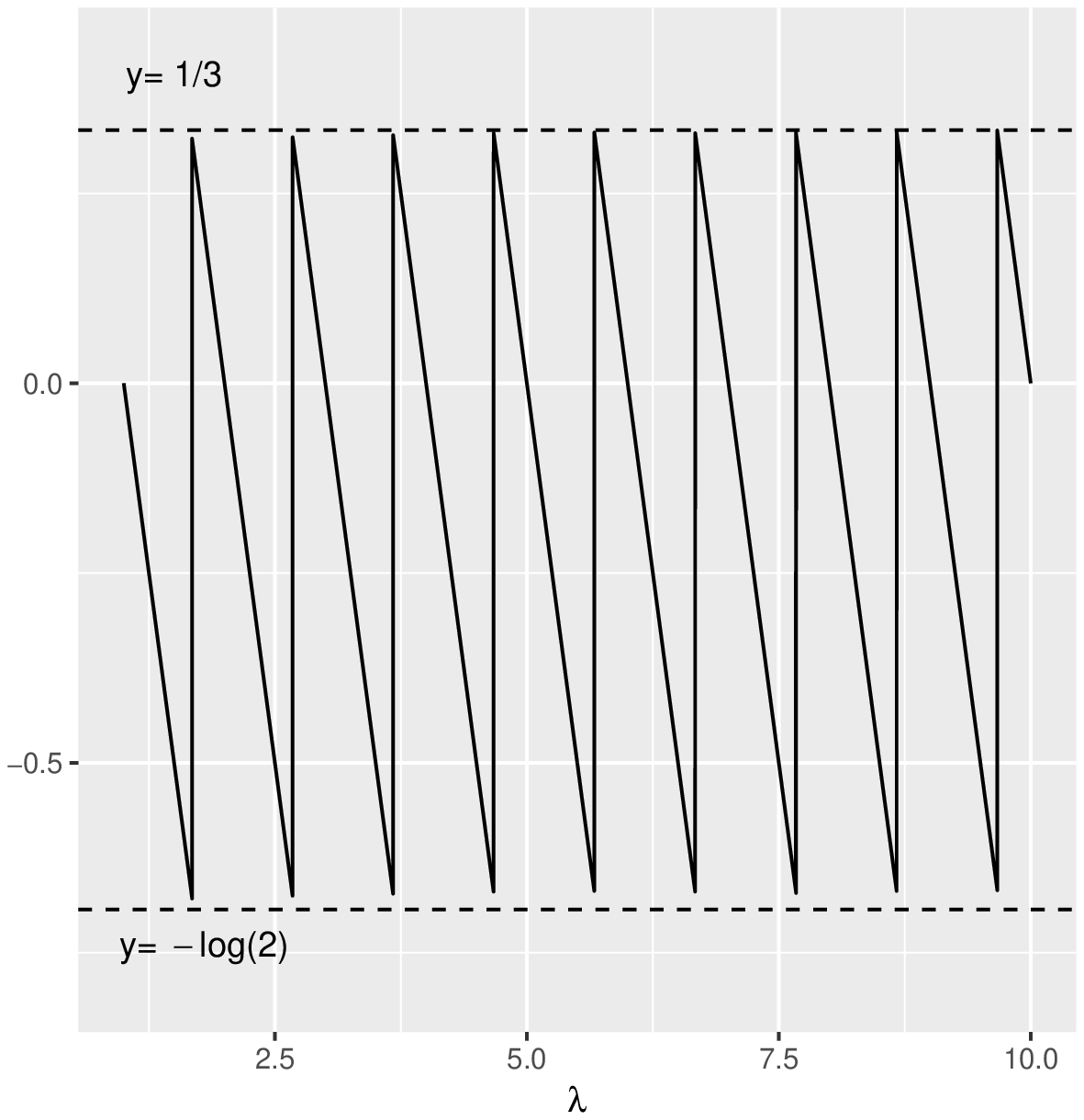}} &  
\subfloat[{$\Me_{Z_\lambda}-\lambda-1/3$, $\lambda\in[0,10]$}]{\includegraphics[scale=.45]{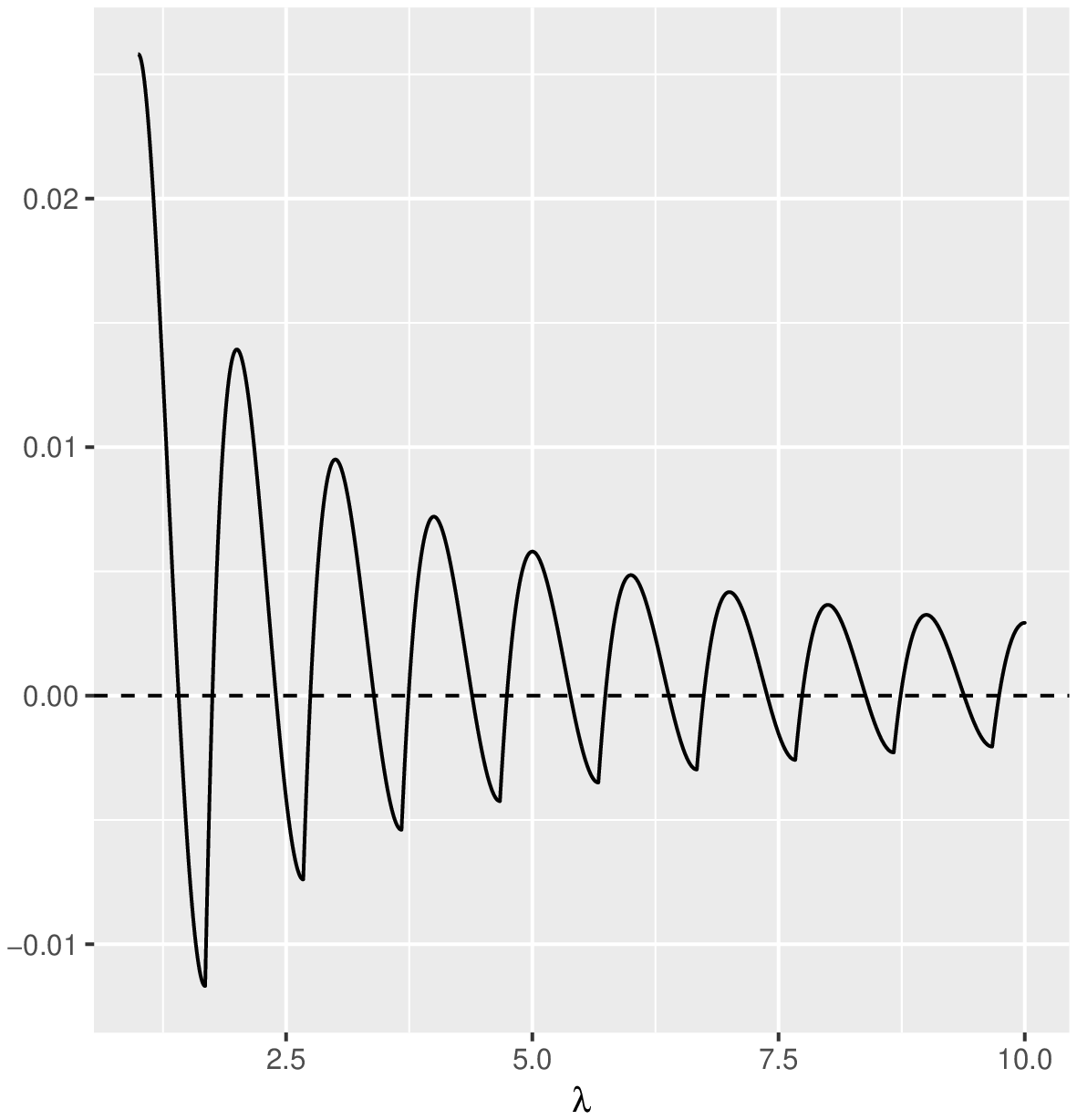}} \\
\subfloat[{$\Me_{Z_\lambda}-\lambda-1/3$, $\lambda \in [10,20]$}]{\includegraphics[scale=.45]{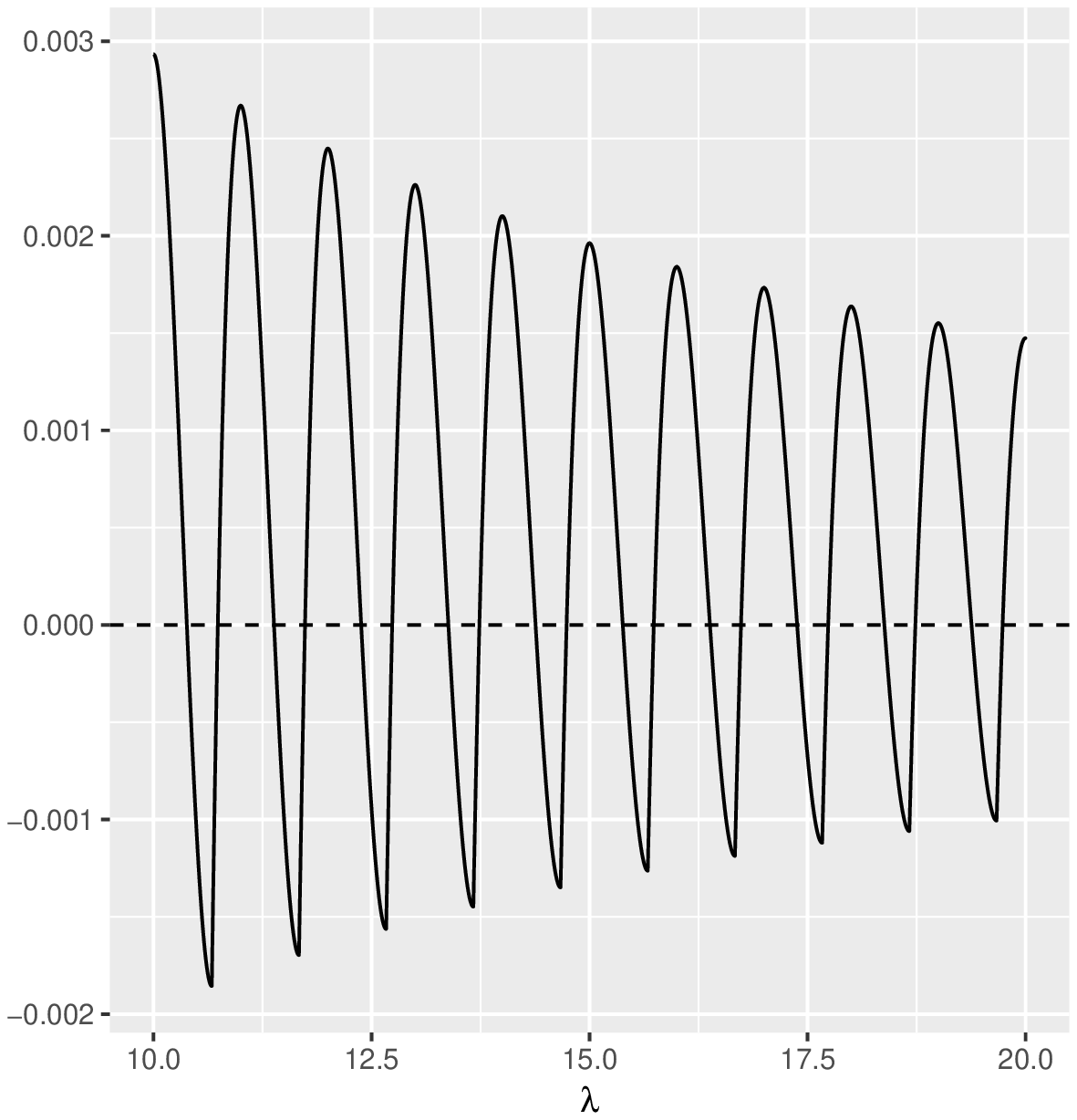}} &  
\subfloat[{$\mathcal H(x)$}]{\includegraphics[scale=.45]{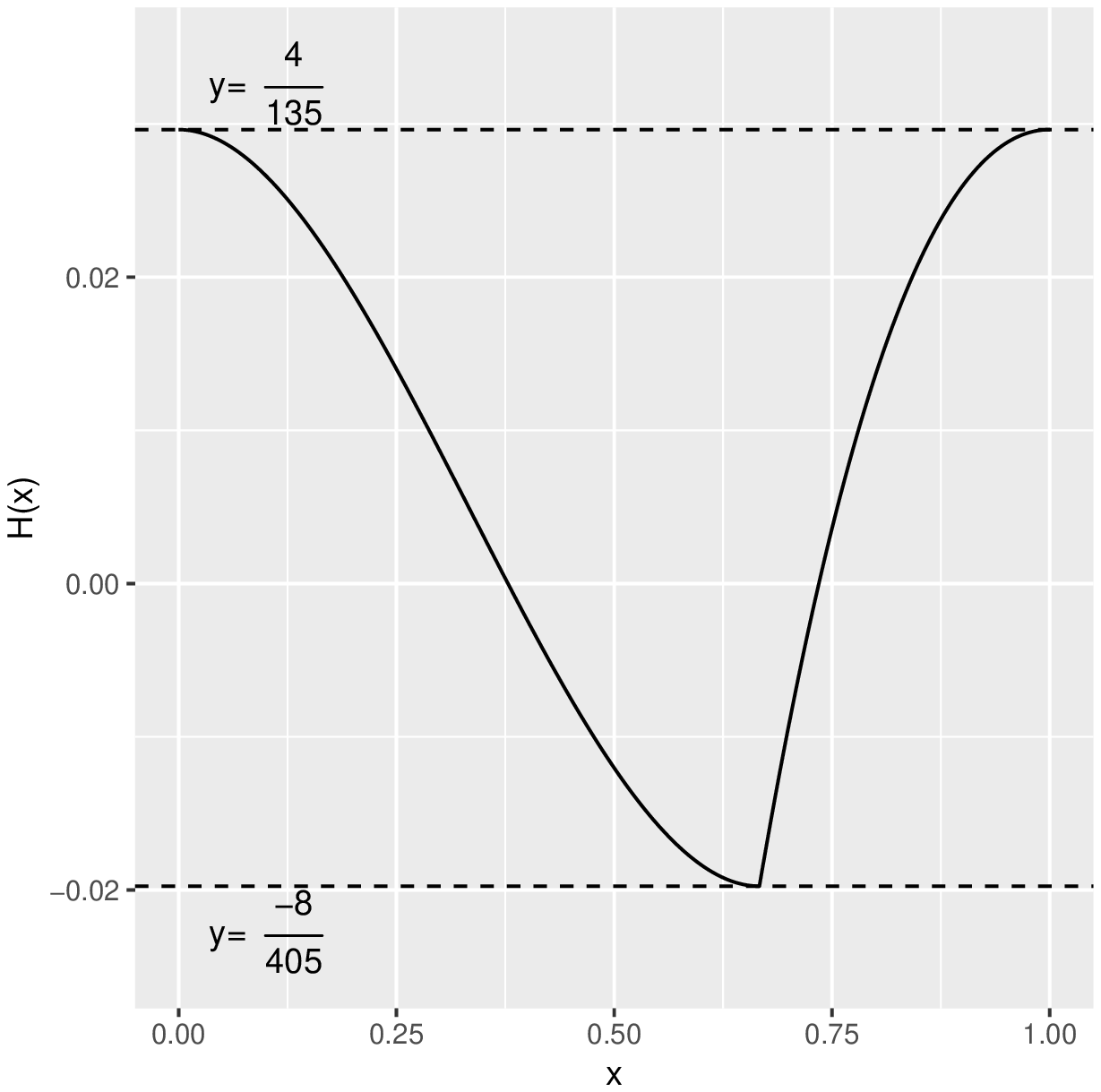}} \\  
\end{tabular}
\caption{\label{fig1} (a)-(c) Plots of $\Me_{N_\lambda}-\lambda$ and $\Me_{Z_\lambda}-\lambda-1/3$ in terms of $\lambda$; (d) Plot of the function $\mathcal H(\cdot)$ given by~\eqref{eq:H}.}
\end{figure}

\begin{theorem}\label{thm:1}
Let $Z_\lambda=N_\lambda+U$ where $N_\lambda$ and $U$ are two independent random variables respectively distributed as a Poisson distribution with parameter $\lambda >0$ and a uniform distribution on $(0,1)$. Then, as $\lambda \to \infty$, the median, $\Me_{Z_\lambda}$ of $Z_\lambda$ satisfies
\begin{equation}\label{eq:thm}
 Me_{Z_\lambda} = \lambda +\frac13 + \frac{\mathcal H(\lambda-\lfloor \lambda \rfloor)}\lambda  +  o \left( \frac1\lambda \right)
\end{equation}
where $\mathcal H:[0,1]\to \R$ is the continous function given by
\begin{equation}\label{eq:H}
  \mathcal H(x) =\left\{
\begin{array}{ll}
\frac{x^2(x-1)}{3}+\frac{4}{135}, & \mbox{ if } x\in [0,2/3] \\
\frac{x}{3}(x^2-4x+5)-\frac{86}{135}, & \text{ if } x\in[2/3,1].
\end{array}
\right.   
\end{equation}
\end{theorem}

Equation~\eqref{eq:liminfsup} is easily deduced since  (see also Figure~\ref{fig1}(d)) we can check that 
\[
  \inf_{x\in[0,1]} \mathcal H(x) = -\frac{8}{405} \quad \text{ and } \quad \sup_{x\in [0,1]} \mathcal H(x) = \frac{4}{135}.
\]
Theorem~1 ensues from the following proposition for which we provide a sketch of the proof. 

\begin{proposition} \label{prop}
Let $n\geq 1$ and $x\in [0,1)$. For any $\varepsilon>0$, there exists $n_0\in \N$ such that for all $n\geq n_0$
\begin{equation}\label{eq:Zn}
\frac{\mathcal H(x)-\varepsilon}{n+x} \le \Me_{Z_{n+x}} -(n+x)-\frac13    \le \frac{\mathcal H(x)+\varepsilon}{n+x}.  
\end{equation}

\end{proposition}

\begin{proof}
Let $k\in \R$, and $(w_n(x,k))_{n\ge 1}$  be the sequence given by $w_n(x,k)= \P(Z_{n+x}\le n+x+1/3+k/(n+x))$ and $V_n= Z_{n+x}-(n+x+1/3+k/(n+x))$. First, 
\[
V_n \stackrel{d}{=}\sum_{i=1}^{n}(N_i-1)+\left(N_x+U-x-1/3-\frac{k}{n+x}\right)
  \]
where $\stackrel{d}{=}$ stands for equality in distribution, $N_1,\ldots,N_n,N_x$ and $U$ are independent random variables, $N_i\sim P(1)$ for $i=1,n$, $N_x\sim P(x)$ and $U\sim \mathcal U([0,1])$. Central limit theorem and Slutsky's lemma show that $ V_n /\sqrt{n} \to N(0,1)$ in distribution as $n\to \infty$, whereby we deduce that for any $x\in [0,1)$ and $k\in \R$, $w_n(x,k) \to 1/2$ as $n\to \infty$. 

Now, assume $x$ and $k$ are such that for $n$ sufficiently large $(w_n(x,k))_{n\ge 1}$ is increasing or decreasing then obviously 
\begin{equation}
    \label{eq:wn1}
    w_n(x,k)\leq w_{n+1}(x,k) \leq \frac{1}{2} \Leftrightarrow \Me_{Z_{n+x}}\geq n+x+\frac{1}{3}+\frac{k}{n+x}
  \end{equation}  
or 
\begin{equation}\label{eq:wn2}
  w_n(x,k)\geq w_{n+1}(x,k) \geq \frac{1}{2} \Leftrightarrow \Me_{Z_{n+x}}\leq n+x+\frac{1}{3}+\frac{k}{n+x}.
\end{equation}

So, the rest of the proof simply consists in proving that the sequence $(w_n(x,k))_{n\ge 1}$ is monotonic for $n$ sufficiently large. We start by noting that the discontinuity at $x=2/3$ of the function $\mathcal H$ comes from the definition of $w_n(x,k)$. Lemma~\ref{lemma1} shows in particular that for $n$ sufficiently large
\begin{itemize}
  \item if $x\in (0,2/3)$ or $x=2/3$ and $k<0$, 
  \[
    w_n(x,k)= \P\left(N_{n+x}\leq n \right)+\left(x-\frac{2}{3}+\frac{k}{n+x}\right)\P\left(N_{n+x}=n\right).
   \] 
   \item if $x\in (2/3,1)$ or $x=2/3$ and $k\ge 0$
   \[
    w_n(x,k)= \P\left(N_{n+x}\leq n\right)+\left(x-\frac{2}{3}+\frac{k}{n+x}\right)\P(N_{n+x}=n+1).
   \]
 \end{itemize} 
Then, we define $\Delta_n(x,k)= (n+1)!/g_{n}(n+1+x) \{w_{n+1}(x,k) - w_n(x,k) \}$ where for any $n\ge 1$ and $u\in \R$, $g_n(u)=\mathrm e^{-u} u^n$. Lemma~\ref{lemma3}, which is based on simple but lengthy Taylor expansions, shows that for any $x \in [0,1)$
\begin{equation*}
  % \label{eq:Delta}
\Delta_n(x,k) = \frac{3}{2(n+1+x)^2} \big( \mathcal H(x) - k \big) \; + \; o\left(\frac1{n^2}\right).
\end{equation*}
So, if we set $k=\mathcal H(x)+\varepsilon$, $(w_n(x,k))_{n\ge 1}$ is a decreasing sequence for $n$ sufficiently large which, from~\eqref{eq:wn2} leads to the upper-bound of~\eqref{eq:Zn}. In the same way, if we set $k=\mathcal H(x)-\varepsilon$, $(w_n(x,k))_{n\ge 1}$ is an increasing sequence for $n$ sufficiently large which, from~\eqref{eq:wn1} leads to the lower-bound of~\eqref{eq:Zn}.
\end{proof}

Our main result has a simple statistical application.We suggest to estimate $\lambda$ by $\hat \lambda^\mathrm{J} = \Mee(\bZ_\lambda) - 1/3$: $\hat \lambda^\mathrm{J}$ is almost an unbiased estimator of $\lambda$, and we can use the approximation
\[
   \hat \lambda^\mathrm{J} - \lambda  \approx N\left( 0 , \frac{\sigma^2_\lambda}n \right),
\]
where $1/\sigma_\lambda = 2 \P(N_\lambda= \lfloor \Me_{Z_\lambda} \rfloor )$. Note that $1/\sigma_\lambda$ can simply be estimated by $1/\hat \sigma_\lambda= 2 \P(N_{\hat \lambda^\mathrm{J}} = \lfloor \hat \lambda^\mathrm{J} +1/3\rfloor)$. When $\lambda$ is large, we can even use Stirling's formula to approximate $\sigma_\lambda \sim \sqrt{\pi\lambda/2} $ which is then simply estimated by $\sqrt{\pi\hat \lambda^\mathrm{J}/2}$.
Therefore, $\sqrt{\pi/2}$ represents the ratio of asymptotic standard deviations of $\hat \lambda^\mathrm{J}$ (when $\lambda$ is large) and the maximum likelihood estimator. We can wonder where this $\sqrt{\pi/2}$ comes from: actually this ratio is also the ratio of standard deviations of the sample median to the sample mean when we consider a sample of i.i.d. Gaussian random variables with mean~0 and variance~1.

We end this section by stressing on the simplicity of the estimator $\hat \lambda^\mathrm{J}$. We do not resist to provide the \texttt{R} instruction to evaluate it based on a sample stored in a vector \texttt{y}: \\
\texttt{> median(y+runif(length(y)))-1/3}

\section{Numerical results} \label{sec:sim}

\subsection{Performances of $\hat \lambda^\mathrm{J}$ without outliers}

For 100 values of $\lambda$ between 1 and 10, we generate 10000 replications of samples of Poisson distribution of size $n$ with parameter $\lambda$. We consider $n=50$ and $n=200$. For each sample $\bN_\lambda$, we evaluate  the maximum likelihood estimate that we denote in the sequel by $\hat \lambda^{\mathrm{MLE}}$, $\hat \lambda^\mathrm{J}$ and $\widehat \Me(\bN_{\lambda})$. Figure~\ref{fig2} reports empirical biases in terms of $\lambda$. As expected, $\hat \lambda^\mathrm{MLE}$ and $\hat \lambda^\mathrm{J}$ are almost unbiased while $\Mee(\bN_\lambda)$ has some bias which doesn't disappear with large $\lambda$ or large $n$. Figure~\ref{fig3} shows $\mathrm{RMSE}(\hat \lambda^\mathrm{J})/\mathrm{RMSE}(\hat \lambda^\mathrm{MLE})$ which is the ratio of the root mean squared error (RMSE) of $\hat \lambda^\mathrm{J}$ to the one of $\hat \lambda^\mathrm{MLE}$. Obviously the MLE outperforms $\hat \lambda^\mathrm{J}$ and we observe that the ratio of RMSE is close to $\sqrt{\pi/2}$ when $\lambda$ gets large. Finally, to confirm the estimation of the standard error of $\hat \lambda^\mathrm{J}$ and its asymptotic normality, we investigate the random variable
\begin{equation}
  \label{eq:Deltal}
  \Delta_\lambda = 2\sqrt{n} (\hat \lambda^\mathrm{J} - \lambda)  \P \left( N_{\hat \lambda^{\mathrm{J}}}= \lfloor \hat \lambda^\mathrm{J} +1/3\rfloor \right)
\end{equation}
for which its distribution should be close to standard normal distribution. For each value of $\lambda$ considered, Figure~\ref{fig4} depicts for $n=50$ and $n=200$ the 100 normal probability plots. Actually, we only represent the fitted linear regression models and the expected theoretical line $y=x$. We conclude that for every value of $\lambda$, $\Delta_\lambda$ seems indeed well-approximated by a $N(0,1)$ distribution.

\begin{figure}[h]
\centering\begin{tabular}{cc}
\subfloat[{$n=50$}]{\includegraphics[scale=.45]{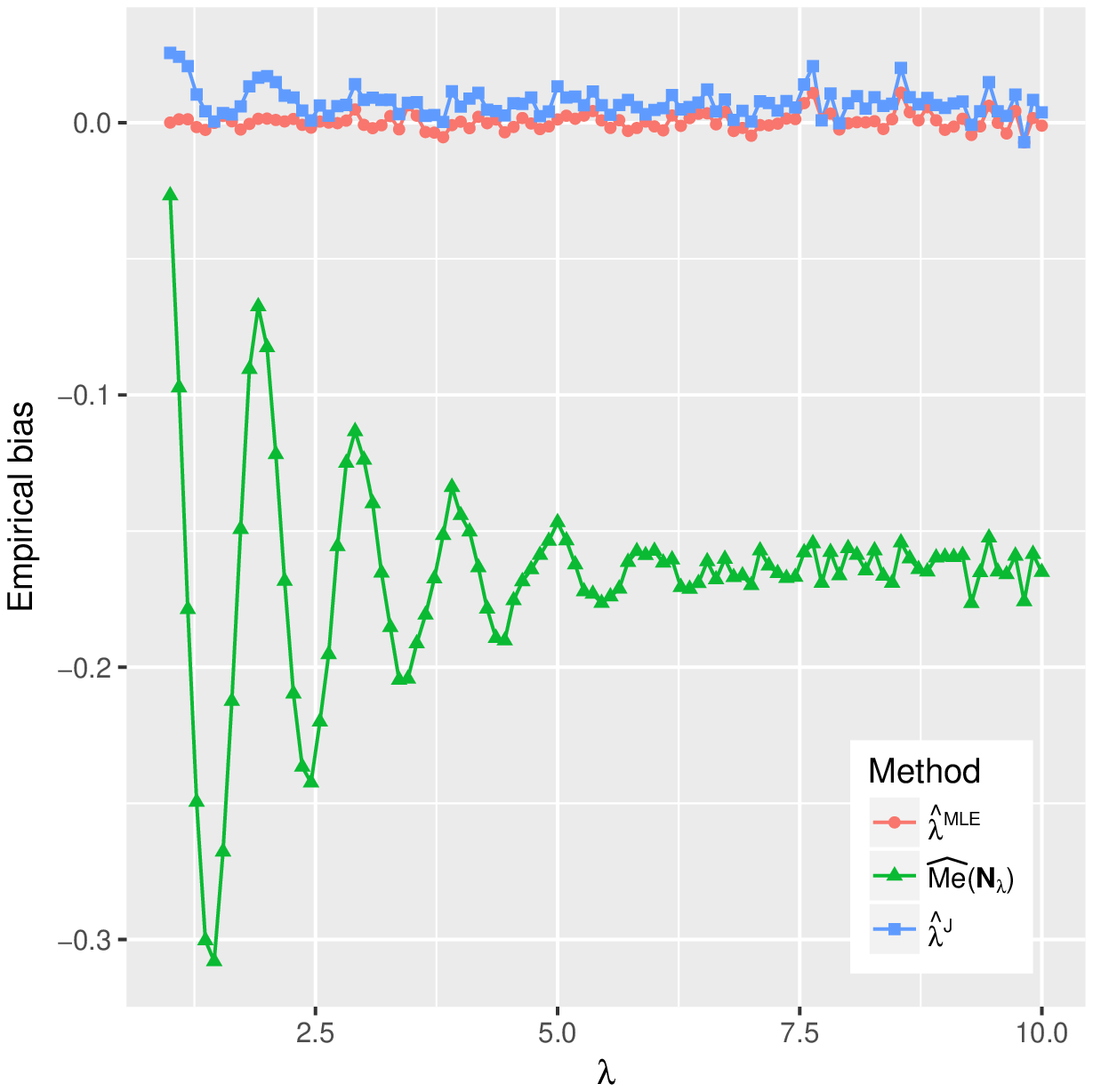}} & 
\subfloat[{$n=200$}]{\includegraphics[scale=.45]{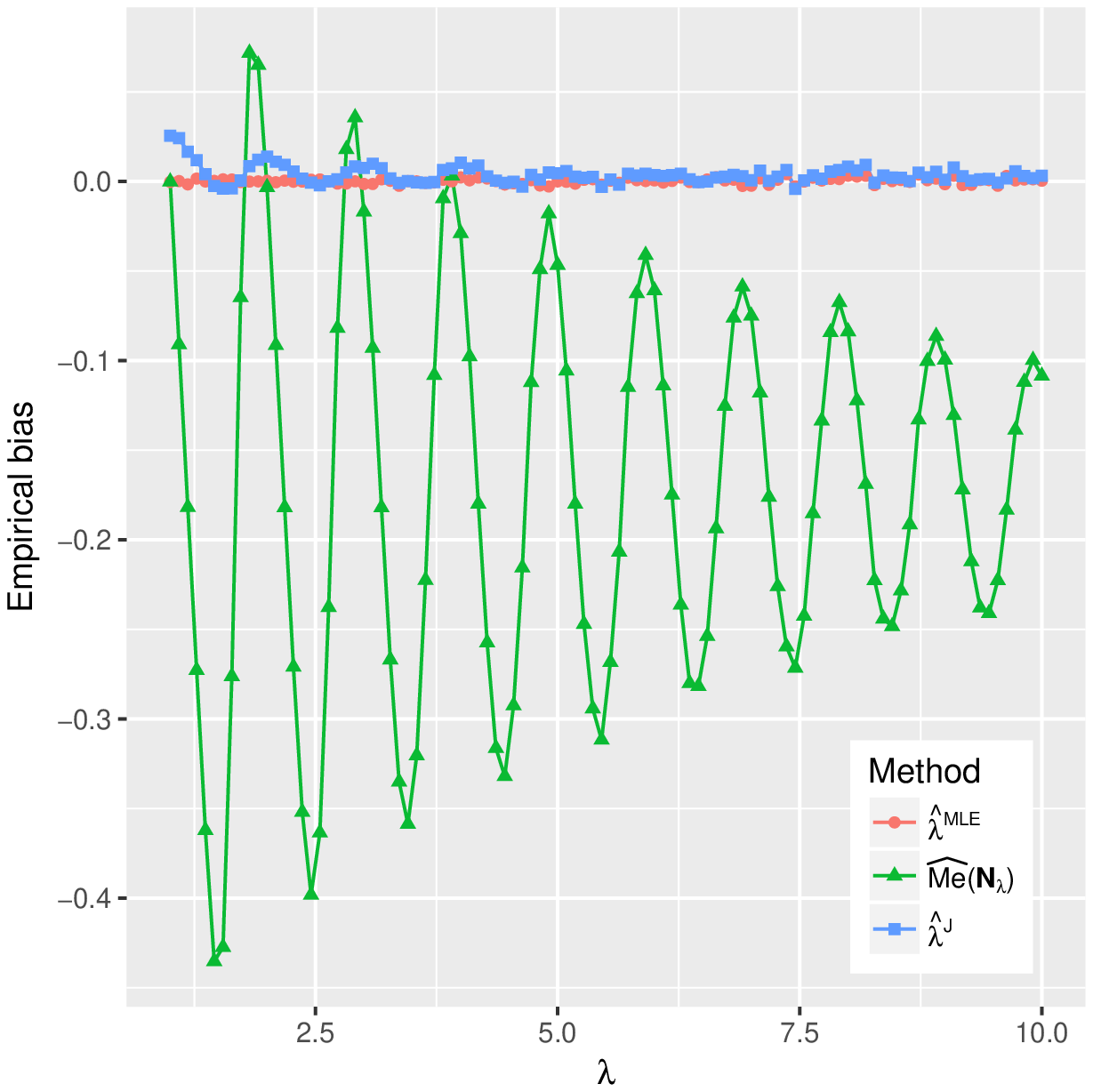}} \\  
\end{tabular}
\caption{\label{fig2} Empirical biases for estimators of $\lambda$ based on $10000$ replications of samples of size $n$ from a Poisson distribution with parameter $\lambda$. 100 values of $\lambda$ between 1 and 10 are considered.}
\end{figure}

\begin{figure}[h]
\centering\begin{tabular}{cc}
\subfloat[{$n=50$}]{\includegraphics[scale=.45]{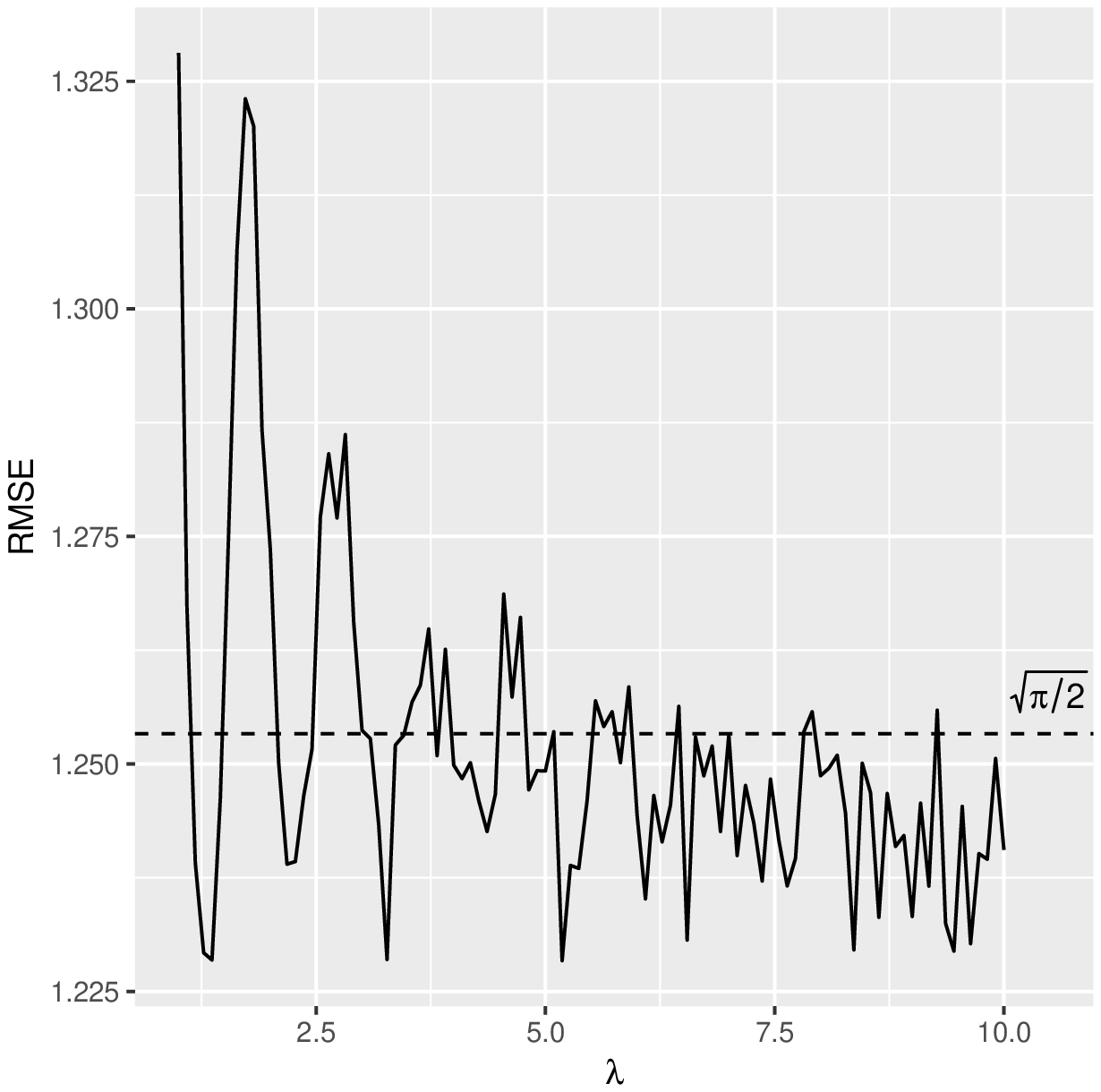}} &  
\subfloat[{$n=200$}]{\includegraphics[scale=.45]{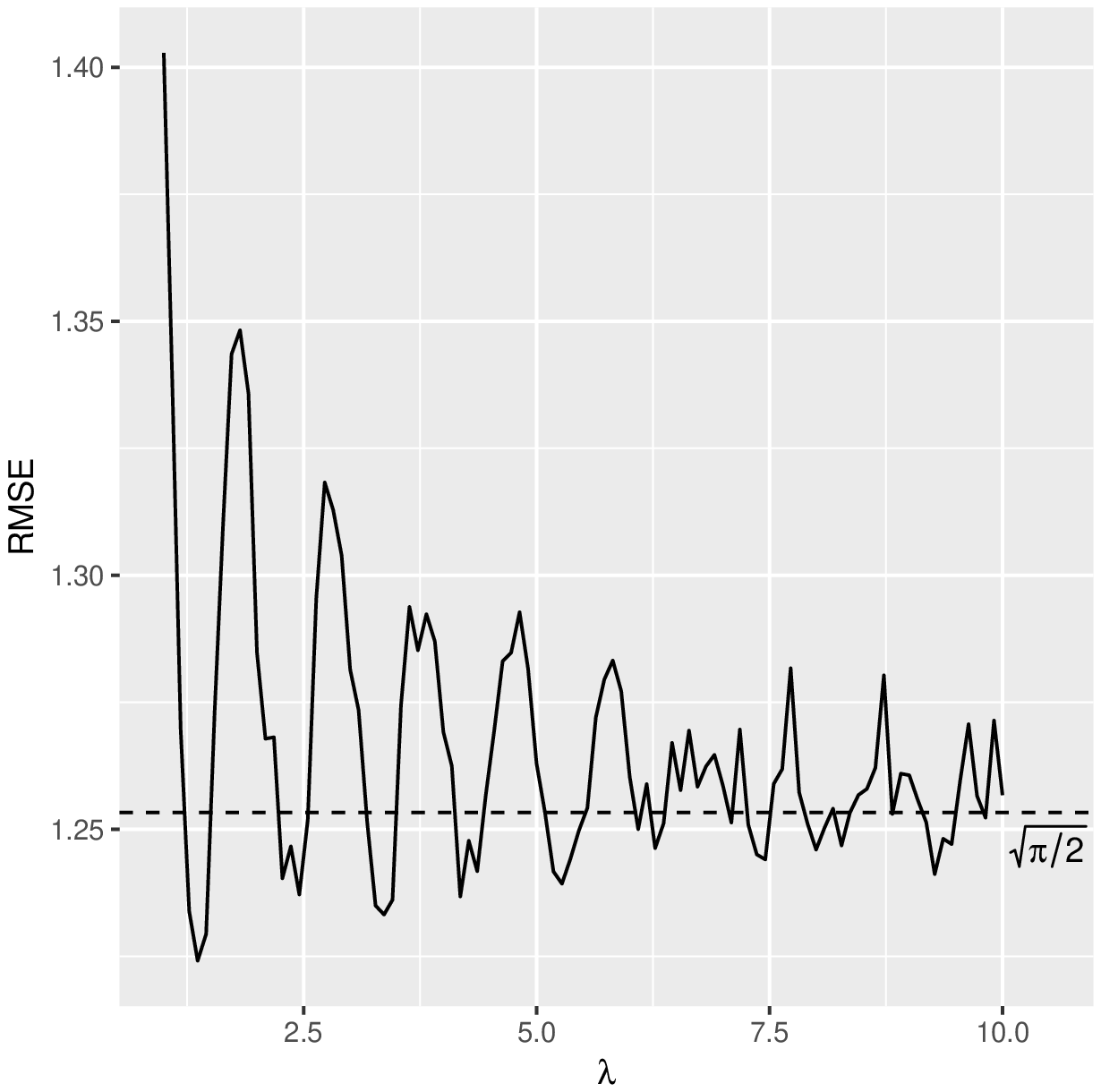}} \\ 
\end{tabular}
\caption{\label{fig3} Ratio of RMSE of the jittered estimate $\hat \lambda^\mathrm{J}$ to the one of the maximum likelihood estimate. The simulation is based, for each $\lambda$, on $10000$ replications of samples of size $n$ from a Poisson distribution with parameter $\lambda$. 100 values of $\lambda$ between 1 and 10 are considered.}
\end{figure}

\begin{figure}[htbp]
\centering\begin{tabular}{cc}
\subfloat[{$n=50$}]{\includegraphics[scale=.45]{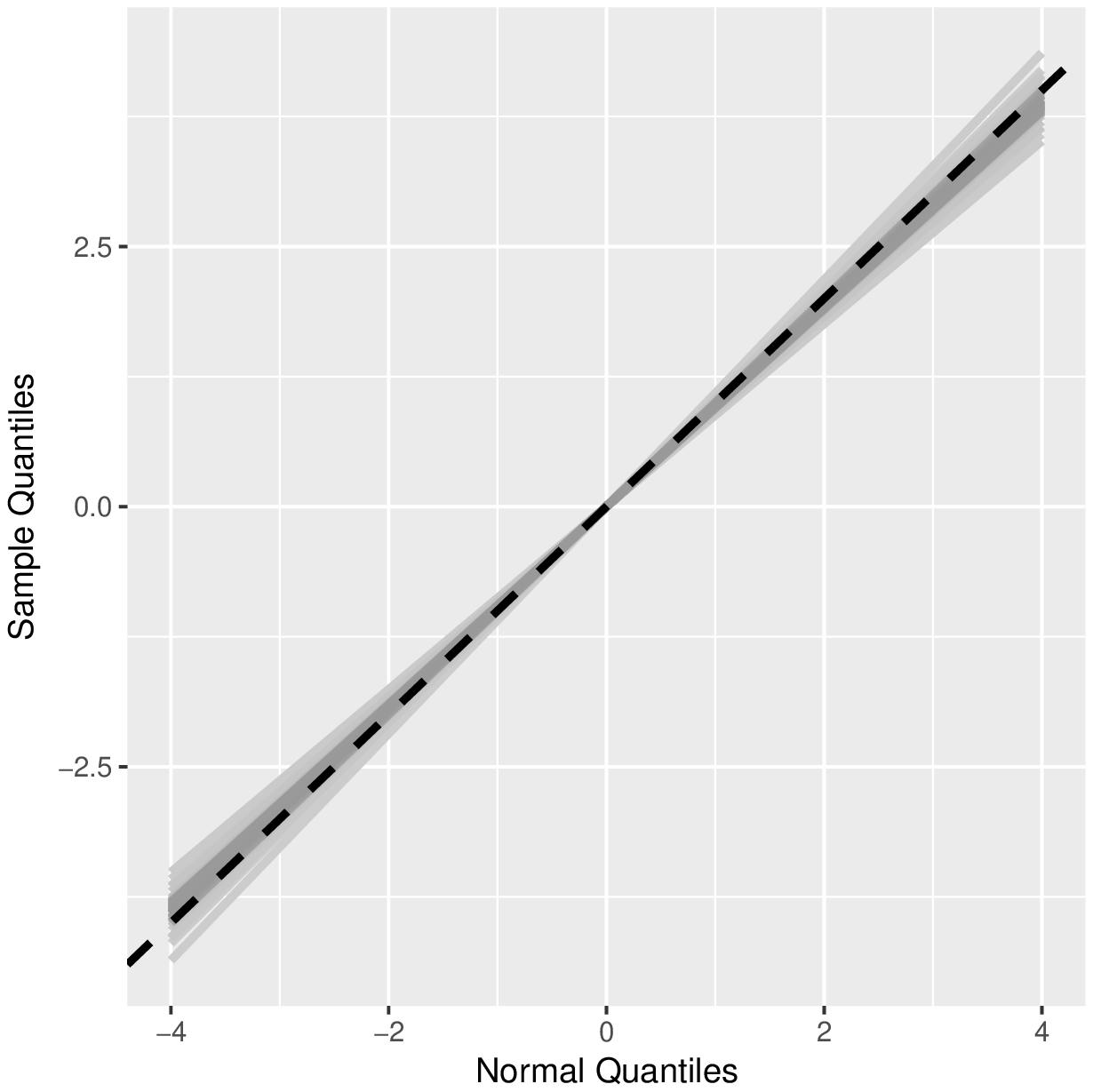}} &  
\subfloat[{$n=200$}]{\includegraphics[scale=.45]{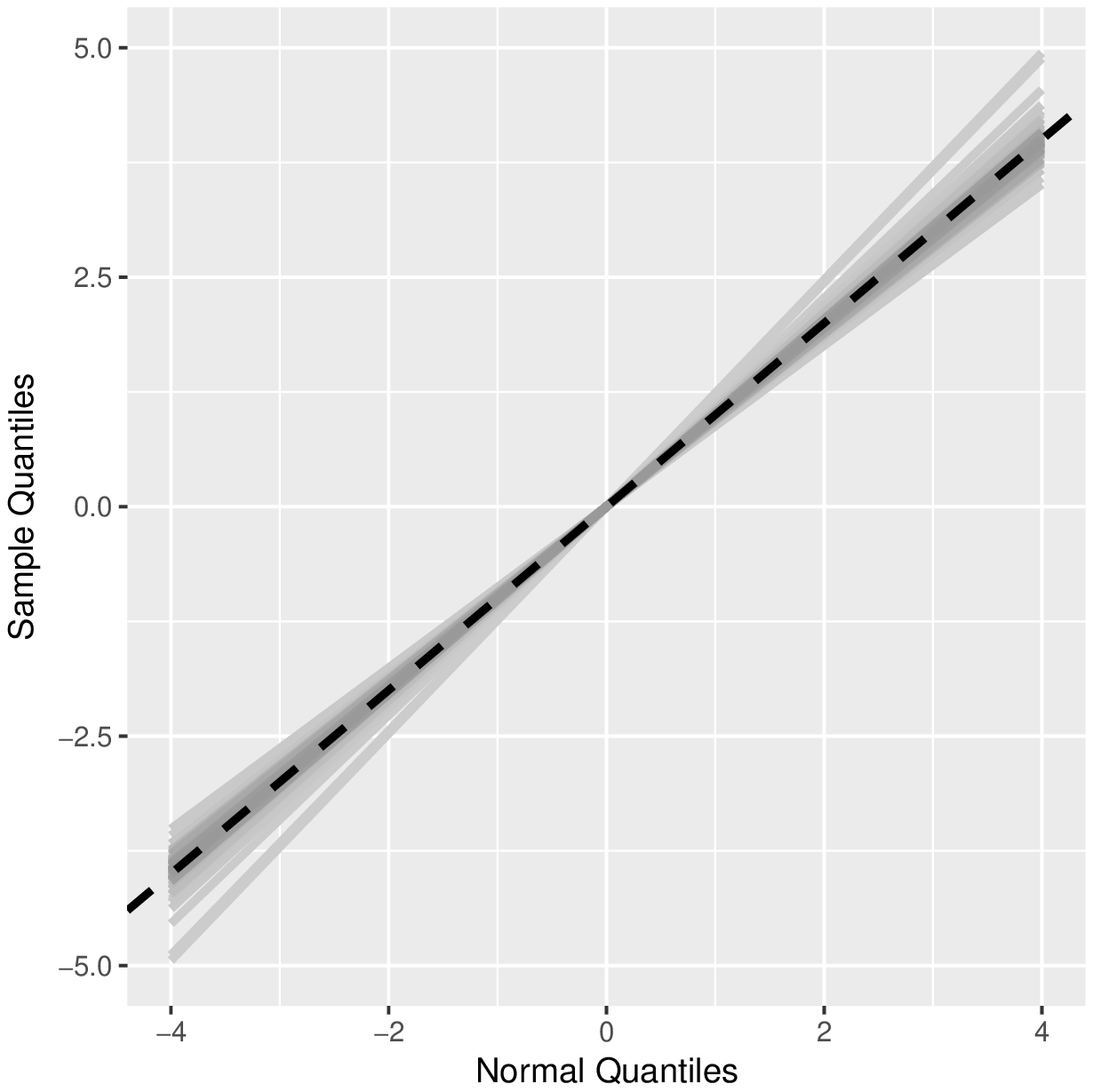}} \\ 
\end{tabular}
\caption{\label{fig4} Fitted lines of normal probability plots obtained for each value of $\lambda$ for the random variable $\Delta_\lambda$ given by~\eqref{eq:Deltal}. The simulation is based, for each $\lambda$, on $10000$ replications of samples of size $n$ from a Poisson distribution with parameter $\lambda$. 100 values of $\lambda$ between 1 and 10 are considered.}
\end{figure}

\subsection{Simulation study in presence of outliers}

In this section, we compare three estimators of $\lambda$: $\hat\lambda^{\mathrm{MLE}}$ which serves as a baseline, $\hat \lambda^\mathrm{J}$ and the Tukey's modifed estimator proposed by~\cite{elsaied:fried:16}. This estimator, denoted by $\hat \lambda^{\mathrm{Tukey}}$ in the sequel, is an $M$-estimator, see e.g. \cite{vandervaart}, with objective function 
\begin{equation}
\psi_{k,a}(y,\lambda)=\left(\frac{y-\lambda}{\sqrt{\lambda}}-a\right)\left(k^2-\left(\frac{y-\lambda}{\sqrt{\lambda}}\right)^2\right)^2\mathrm{I}_{[-k,k]}\left(\frac{y-\lambda}{\sqrt{\lambda}}-a\right)
\label{Tukeymod}
\end{equation}
where $a=a(\lambda,k)$ is such that $\mathrm{E}(\psi_{k,a}(Y,\lambda))=0.$ \cite{elsaied:fried:16} considered several other estimators (other version of $M$-type estimators, weighted likelihood type estimators, etc), made an extensive simulation study and concluded that in many situations ${\hat \lambda}^{\mathrm{Tukey}}$ was the best one. To tune the constant $k$ and thus the corrective term $a=a(\lambda,k)$, we follow the suggestion by~\cite{elsaied:fried:16} and set the constant $k=6$. Given a first estimate of $\lambda$, a first corrective term $a$ is found which serves as a first $M$-estimation and so on. The algorithm is stopped when the difference between two successive estimates of $\lambda$ does not exceed $10^{-4}$.

The simulation model we consider is an additive outliers type model  where we assume to observe $\widetilde N_{i,\lambda}$ given by  
\begin{equation}\label{eq:AO}
\widetilde N_{i,\lambda} = N_{i,\lambda} + (1-\varepsilon_i) \sqrt h, \qquad  i=1,\dots,n,  
\end{equation}
 where $P(\varepsilon_i=1)=1-\pi$ and where $\pi$ corresponds to the proportion of outliers and $h$ is a constant. For a given $\pi$, we consider the signal-to-noise ratio defined in decibels by $SNR=10\log_{10}(\frac{\lambda}{h\pi(1-\pi)})$ and set $\sqrt h$ (as an integer value) such that $\mathrm{SNR}=-10$ (db). 
    
Figure~\ref{figAO} reports empirical biases and RMSE in terms of \linebreak $\pi\in\{0\%,1\%,5\%, 10\%, 20\%\}$ for estimators of $\lambda=5$. These Monte-Carlo results are based on 10000 replications from the model~\eqref{eq:AO} with $n=50$ and $n=200$. As expected, the MLE gets quickly biased as soon as $\pi>0$ which makes its RMSE very high. When $\pi$ is not too large, $\hat \lambda^\mathrm{J}$ seems less biased than $\hat \lambda^\mathrm{Tukey}$. However, the latter one is more efficient which explains why the RMSE of the $\hat \lambda^\mathrm{Tukey}$ is smaller than the one of $\hat \lambda^\mathrm{J}$.  It is to be noticed that this difference tends to decrease when $n$ increases. When $\pi=10\%$ or $20\%$, $\hat \lambda^\mathrm{Tukey}$ is much less biased than $\hat \lambda^\mathrm{J}$ and still has a smaller variance. 

As a conclusion of this simulation study, it turns out that the naive and very simple estimator $\hat \lambda^\mathrm{J}$ behaves nicely compared to very efficient estimators such as the Tukey's modified estimator, when the proportion of outliers is low. When this proportion increases, the performances of the median-based estimator $\hat \lambda^\mathrm{J}$ degrade, whereas $\hat \lambda^\mathrm{Tukey}$ still remains efficient.
\begin{figure}[htbp]
\centering\begin{tabular}{cc}
\subfloat[{$n=50$}]{\includegraphics[scale=.45]{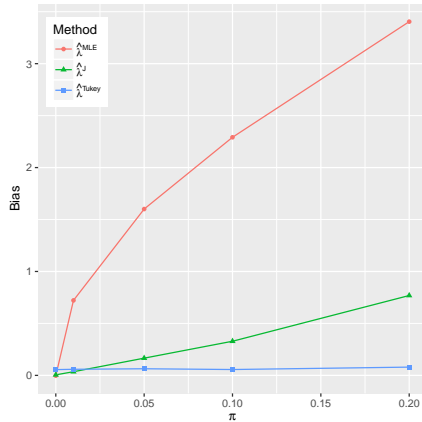}} & 
\subfloat[{$n=200$}]{\includegraphics[scale=.45]{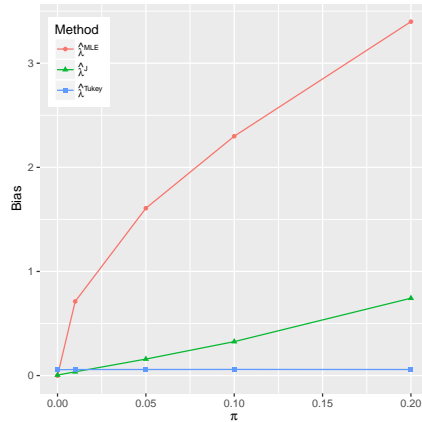}} \\  
\subfloat[{$n=50$}]{\includegraphics[scale=.45]{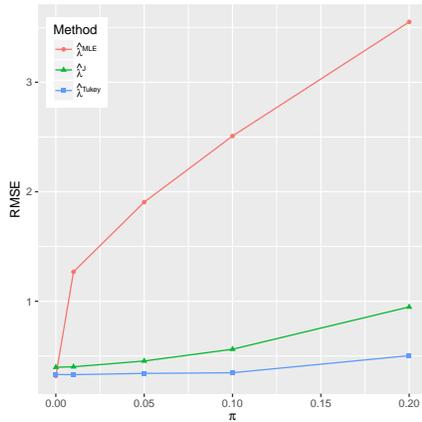}} & 
\subfloat[{$n=200$}]{\includegraphics[scale=.45]{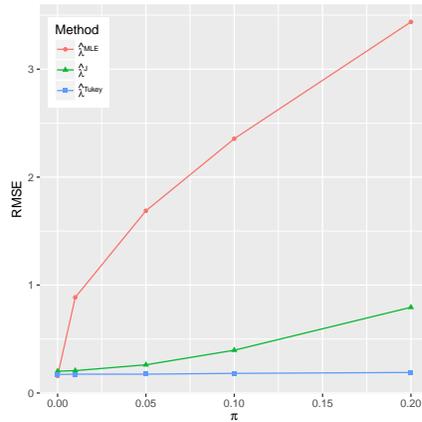}} \\  
\end{tabular}
\caption{\label{figAO} Empirical biases ((a)-(b)) and RMSE ((c)-(d)) for estimators of $\lambda=5$ based on $10000$ replications of samples of size $n=50$ and 
$n=200$ from the model~\eqref{eq:AO}. Results are reported in terms of $\pi$ the proportion of outliers.}
\end{figure}

\subsection{Computational time}

Table~\ref{tab:time} reports average computational time of four estimators of $\lambda$ ($\hat \lambda^\mathrm{MLE}$, $\Mee(\bN_\lambda)$, $\hat \lambda^\mathrm{J}$ and $\hat \lambda^\mathrm{Tukey}$) in terms of the sample size.  The estimates are implemented in \texttt{R} on a laptop using a 2,9 GHz Intel Core i5 process.  The MLE is obviously the cheapest one. Evaluating $\hat \lambda^\mathrm{J}$ is approximately twice more expensive than evaluating $\Mee(\bN_\lambda)$. This factor is due to the generation of uniform distributions. However, the computational time is very reasonable even for very large datasets compared to the Tukey's modified estimate which is unfeasible for $n=10^9$ due to memory storage.

\begin{table}[htbp]
\begin{center}
\begin{tabular}{lcccccc}
\hline
& \multicolumn{5}{c}{Sample size $n$} \\
& $10^4$ &$10^5$ &$10^6$ &$10^7$ &$10^8$ & $10^9$\\
\hline
$\hat \lambda^\mathrm{MLE}$ &    $0^\star$ & $0^\star$ &$0^\star$  &$0^\star$ &   0.1  & 1.2\\
$\Mee(\bN_\lambda)$ &  $0^\star$ &   $0^\star$ &   $0^\star$ &0.2& 2.8 & 20.1\\
$\hat \lambda^\mathrm{J}$ &$0^\star$ &$0^\star$ &0.1 & 0.5 & 5.5 & 48.2\\
$\hat \lambda^\mathrm{Tukey}$& $0^\star$ &0.1 & 0.8 & 8.2 & 98.2 & NA \\
\hline
\end{tabular}
\caption{\label{tab:time}Average computational time in seconds (based on 10 replications) to evaluate each estimate based on a sample of size $n$ from a Poisson distribution with parameter $\lambda=\pi$. The notation $0^\star$ means that the average time is smaller than 0.05 sec.} 
\end{center}
\end{table}

\section*{Acknowledgements}

The authors are 
sincerely grateful to H. Elsaied and R. Fried for discussions and for sharing the R code implementing the Tukey's modified $M$-estimator. The research of J.-F. Coeurjollly is supported by the Natural Sciences and Engineering Research Council of Canada.

% \bibliographystyle{abbrvnat}
% \setcitestyle{authoryear,open={((},close={))}}
\bibliographystyle{plain}
% \bibliographystyle{plain}      % basic style, author-year citations

% \begin{thebibliography}{}
% \bibitem{adell:jodra} Adell, J. A., \& Jodr\'{a}, P. (2005). 
% The median of the Poisson distribution.
% {\it Metrika}, 61(3), 337--346.

\bigskip

\appendix

\section{Technical lemmas}

Appendix gathers technical lemmas, used in the proof of Proposition~1.

\begin{lemma}
  \label{lemma1}
  Let $k\in \R$, $x\in [0,1)$ and $r_n(x,k) = k/(n+x)$. Then the sequence $w_n(x,k)$ reads as follows:\\
  (i) If $x+r_n(x,k) \in [-1/3,2/3)$
  \begin{equation}
  \label{eq:wn_case1}
  w_n(x,k)= \P\left(N_{n+x}\leq n \right)+\left(x-\frac{2}{3}+\frac{k}{n+x}\right)\P\left(N_{n+x}=n\right).
  \end{equation}
  (ii) If $x+r_n(x,k)\in [2/3,5/3)$
  \begin{equation}\label{eq:wn_case2}
  w_n(x,k)= \P\left(N_{n+x}\leq n\right)+\left(x-\frac{2}{3}+\frac{k}{n+x}\right)\P(N_{n+x}=n+1).   
  \end{equation} 
\end{lemma}
Let $k\in \R$, then for $n$ sufficiently large, if $x\in [0,2/3)$ or $x=2/3$ and $k<0$, then $w_n(x,k)$ reads as in~\eqref{eq:wn_case1}. In the same way, if $x\in (2/3,1)$ or $x=2/3$ and $k\ge 0$, then $w_n(x,k)$ reads as in~\eqref{eq:wn_case2}. The proof of Lemma~\ref{lemma1} is omitted as it derives easily from~\eqref{eq:FZlambda}.

\begin{lemma}
  \label{lemma2}
  Let $k\in \R$, $x\in [0,1)$ and $r_n(x,k) = k/(n+x)$. Let $\Delta_n(x,k)$ be given by
  \[
  \Delta_n(x,k) = \frac{(n+1)!}{g_{n+1}(n+1+x)} \bigg( w_{n+1}(x,k) - w_{n}(x,k) \bigg).
  \]
  There exists $n_0 \in \N$ such that for all $n\geq n_0$, we have the two following cases:\\
  (i) If $x+r_{n}(x,k) \in [-1/3,2/3)$
  \begin{align}
  \Delta_n(x,k) =& \;c_n(0,x)-\int_0^1 c_n(v,x)\dd v + \left(x-\frac23 \right)\left( 1-c_n(0,x) \frac{n+1}{n+x}\right) \nonumber \\
  &+ r_{n+1}(x,k)- \frac{n+1}{n+x} c_n(0,x)r_n(x,k) \label{eq:Delta1_case1}
  \end{align}
  where $c_n(\cdot,x):[0,1]\to \R$ is defined by 
  \[
  c_n(v,x)=  \left( \frac{n+v+x}{n+1+x} \right)^{n+1} \exp(1-v).  
  \]
  (ii) If $x+r_{n}(x,k) \in [2/3,5/3)$
  \begin{align}
  \Delta_n(x,k) =& c_n(0,x) -\int_0^1 c_n(v,x)\dd v +\left(x-\frac23 \right)\left( \frac{n+1+x}{n+2}-c_n(0,x) \right) \nonumber \\
  &+ \frac{n+1+x}{n+2}r_{n+1}(x,k)-  c_n(0,x)r_n(x,k). \label{eq:Delta1_case2}
  \end{align}
\end{lemma}

\begin{proof}%{Proof of Lemma~\ref{eq:Delta1}}{}
  
  (i) Using the Poisson-Gamma relation $\P(N_\lambda \le n)=\frac{1}{n!}\int_\lambda^\infty g_n(u)\mathrm{d}u$ with $g_n(u)=\mathrm{e}^{-u}u^n$ and Lemma~\ref{lemma1}~(i), we can rearrange the difference $w_{n+1}(x,k)-w_n(x,k)$ as
\begin{align*}
  w_{n+1}(x,k)-w_n(x,k)
  &=\frac{1}{(n+1)!}\int_{n+x}^{n+1+x}-g_{n+1}(u)du+\frac{g_{n+1}(n+x)}{(n+1)!}\\
  &\qquad+\left(x-\frac{2}{3}+r_{n+1}(x,k)\right)\P\left(N_{n+1+x}=n+1\right)\\
  &\qquad-\left(x-\frac{2}{3}+r_n(x,k)\right)\P\left(N_{n+x}=n\right),
\end{align*}
which leads to the result after little algebra by noticing that
\[
\P(N_{n+x}=n)=\frac{(n+1)(n+x)^n\mathrm{e}^{-(n+x)}}{(n+1)!}=\frac{(n+1)g_n(n+x)}{(n+1)!}.  
\]
(ii) Using the Poisson-Gamma relation and Lemma~\ref{lemma1}~(ii), we can rearrange the difference $w_{n+1}(x,k)-w_n(x,k)$ as 
\begin{align*}
  w_{n+1}(x,k)-w_n(x,k)
  =&\frac{1}{(n+1)!}\int_{n+1+x}^{\infty}\mathrm{e}^{-u}u^{n+1}du-\frac{1}{(n+1)!}\int_{n+x}^{\infty}\mathrm{e}^{-u}u^{n+1}du\\
  &+\left(x-\frac{2}{3}+r_{n+1}(x,k)\right)\P\left(N_{n+1+x}=n+2\right)\\
  &-\left(x-\frac{2}{3}+r_n(x,k)\right)\P\left(N_{n+x}=n+1\right),
  \end{align*}
which leads to the result after little algebra by noticing that
  \begin{equation*}
  \P(N_{n+1+x}=n+2)=\frac{1}{(n+1)!}g_{n+1}(n+1+x)\frac{n+1+x}{(n+2)}
  \end{equation*}
  and
  \begin{equation*}
  \P(N_{n+x}=n+1)=\frac{1}{(n+1)!}g_{n+1}(n+1+x)c_n(0,x).
  \end{equation*}
\end{proof}

\begin{lemma}\label{lemma3}
  Let $k\in \R$, $x\in [0,1)$ and $r_n(x,k) = k/(n+x)$, for any $k \in \R$ and $x\in [0,1)$, then for $n$ sufficiently large we have
  \begin{equation*}
  \Delta_n(x,k) = \frac{3}{2(n+1+x)^2} \big( \mathcal H(x) - k \big) \; + \; o\left(\frac1{n^2}\right)
  \end{equation*}
  where $\mathcal H$ is the function given by~\eqref{eq:H}. 
\end{lemma}

\begin{proof}
Let $k\in \R$, then for any $x\in [0,1)$, there exists $n_0\in \N$ such that for $n\ge n_0$, either (i) $x+r_{n}(x,k)\in[-1/3,2/3)$ or (ii) $x+r_n(x,k)\in[2/3,5/3)$. In the sequel, we consider both cases and expand the expression of $\Delta_n(x,k)$ given by Lemma~\ref{lemma2}. The following expansions extensively make use of auxiliary results gathered in Lemma~\ref{lemma4}.  \\
(i) Case $x+r_{n}(x,k)\in[-1/3,2/3)$.
\begin{align*}
\Delta_n(x,k)
&=\frac{1}{n+1+x}\left(-\frac{x}{2}+\frac{1}{6}\right)+\frac{1}{(n+1+x)^2}\left(-\frac{x^2}{6}-\frac{x}{24}+\frac{7}{120}\right)\\
&\qquad+\frac{1}{n+1+x}\left(x-\frac{1}{2}\right)+\frac{1}{(n+1+x)^2}\left(\frac{x^2}{2}-\frac{5}{24}\right)\\
&\qquad+\left(x-\frac{2}{3}\right)\left[\frac{-1/2}{n+1+x}+\frac{1}{(n+1+x)^2}\left(\frac{x^2}{2}-\frac{x}{2}-\frac{7}{24}\right)\right]\\
&\qquad-\frac{3k}{2}\frac{1}{(n+1+x)^2}+o\left(\frac{1}{n^2}\right)\\
&=\frac{3}{2(n+1+x)^2}\left(\frac{x^2(x-1)}{3}+\frac{4}{135}-k\right)+o\left(\frac{1}{n^2}\right)\\
&=\frac{3}{2(n+1+x)^2} \big( \mathcal H(x) - k \big) \; + \; o\left(\frac1{n^2}\right). 
\end{align*}
(ii) Case $x+r_{n}(x,k)\in[2/3,5/3)$.
\begin{align*}
\Delta_n(x,k)
&=\frac{1}{n+1+x}\left(-\frac{x}{2}+\frac{1}{6}\right)+\frac{1}{(n+1+x)^2}\left(-\frac{x^2}{6}-\frac{x}{24}+\frac{7}{120}\right)\\
&\qquad+\frac{1}{n+1+x}\left(x-\frac{1}{2}\right)+\frac{1}{(n+1+x)^2}\left(\frac{x^2}{2}-\frac{5}{24}\right)\\
&\qquad+\left(x-\frac{2}{3}+r_{n+1}(x,k)\right)\left(\frac{n+1+x}{n+2}\right)\\
&\qquad-\left(x-\frac{2}{3}+r_n(x,k)\right)c_n(0,x)\\
&=\frac{1}{(n+1+x)^2}\left(\frac{x^3}{2}-2x^2+\frac{5x}{2}-\frac{43}{45}-\frac{3k}{2}\right)+o\left(\frac{1}{n^2}\right)\\
&=\frac{3}{2(n+1+x)^2}\left(\frac{x^3}{3}-\frac{4x^2}{3}+\frac{5x}{3}-\frac{86}{135}-k\right)+o\left(\frac{1}{n^2}\right)\\
&=\frac{3}{2(n+1+x)^2} \big( \mathcal H(x) - k \big) \; + \; o\left(\frac1{n^2}\right).
\end{align*}
\end{proof}

\begin{lemma}\label{lemma4}
  Let $v,x \in [0,1]$, then we have the following expansions as $n \to \infty$:\\
  (i) 
  \begin{align}
  c_n(v,x)=&1+\frac{1}{n+1+x}\left(x(1-v)-\frac{(1-v)^2}{2}\right)\nonumber\\
  &+\frac{1}{(n+1+x)^2}\left(\frac{(1-v)^2}{2}(x^2+x)-(1-v)^3\left(\frac{x}{2}+\frac{1}{3}\right)
  +\frac{(1-v)^4}{8}\right) \nonumber\\
  &+o\left(\frac{1}{n^2}\right).  \label{eq:cnv}
  \end{align}
  (ii)
  \begin{align}
  c_n(0,x) &=1+\frac{1}{n+1+x}\left(x-\frac{1}{2}\right)+\frac{1}{(n+1+x)^2}\left(\frac{x^2}{2}-\frac{5}{24}\right)+o\left(\frac{1}{n^2}\right).  \label{eq:cn0}
  \end{align}
  (iii)
  \begin{align}
  \int_0^1c_n(v,x)dv=&1+\frac{1}{n+1+x}\left(\frac{x}{2}-\frac{1}{6}\right) \nonumber \\
  &+\frac{1}{(n+1+x)^2}\left(\frac{x^2}{6}+\frac{x}{24}-\frac{7}{120}\right)+o\left(\frac{1}{n^2}\right).  \label{eq:intcnv}
  \end{align}
  (iv) 
  \begin{equation}\label{eq:iv}
  1-c_n(0,x)\frac{n+1}{n+x}=\frac{-1/2}{n+1+x}+\frac{1}{(n+1+x)^2}\left[\frac{x^2}{2}-\frac{x}{2}-\frac{7}{24}\right]+o\left(\frac{1}{n^2}\right).  
  \end{equation}
  (v)
  \begin{equation}\label{eq:v}
  r_{n+1}(x,k)- \frac{n+1}{n+x} c_n(0,x)r_n(x,k)=\frac{-3k}{2}\frac{1}{(n+1+x)^2}+o\left(\frac{1}{n^2}\right).
  \end{equation}
  (vi)
  \begin{equation} \label{eq:vi}
  \frac{n+1+x}{n+2}-c_n(0,x)= \frac{1/2}{n+1+x} + \frac{1}{(n+1+x)^2} \left( \frac{x^2}2-2x -\frac{5}{24}\right) +o\left( \frac1{n^2} \right).
  \end{equation}
  (vii)
  \begin{equation} \label{eq:vii}
  \frac{n+1+x}{n+2}r_{n+1}(x,k)-  c_n(0,x)r_n(x,k) = 
  \frac{-3k}{2}\frac{1}{(n+1+x)^2}+o\left(\frac{1}{n^2}\right). 
  \end{equation}
  \label{resultatsintermediaires}
\end{lemma}

\begin{proof}
(i) Using Taylor expansions, we have 
  \begin{align*}
  c_n(v,x)&=\left(\frac{n+x+v}{n+1+x}\right)^{n+1}\mathrm{e}^{1-v}\\
  &=\exp\left\{(n+1)\log\left(\frac{1-v}{n+1+x}\right)\right\}\exp(1-v)\\
  &=\exp\left\{(n+1)\left[-\frac{(1-v)}{n+1+x}-\frac{(1-v)^2}{2(n+1+x)^2}\right.\right.\\
  &\qquad\left.\left.-\frac{(1-v)^3}{3(n+1+x)^3}+o\left(\frac{1}{n^3}\right)\right]\right\}\exp(1-v)\\
  &=\exp\left\{(1-v)-\frac{(n+1)(1-v)}{n+1+x}\right\}\\
  &\qquad\times\exp\left\{-\frac{(1-v)^2}{2(n+1+x)^2}-\frac{(1-v)^3}{3(n+1+x)^3} \right\}\left(1+o\left(\frac{1}{n^2}\right)\right)\\
  &=\exp\left\{(1-v)\left[1-\frac{(n+1)}{n+1+x}\right]\right\}\exp\left\{-\frac{(1-v)^2}{2(n+1+x)^2} \right\}\\
  &\qquad\times\exp\left\{-\frac{(1-v)^3}{3(n+1+x)^3} \right\}\left(1+o\left(\frac{1}{n^2}\right)\right)\\
  &=\left(1+\frac{(1-v)x}{n+1+x}+\frac{(1-v)^2x^2}{2(n+1+x)^2}\right)\\
  &\qquad\times\left(1-\frac{(1-v)^2(n+1)}{2(n+1+x)^2}+\frac{(1-v)^4(n+1)^2}{8(n+1+x)^4}\right)\\
  &\qquad\times\left(1-\frac{(1-v)^3(n+1)}{3(n+1+x)^3}\right)\left(1+o\left(\frac{1}{n^2}\right)\right)\\
  &=1+\frac{1}{n+1+x}\left((1-v)x-\frac{(1-v)^2}{2}\frac{n+1}{n+1+x}\right)\\
  &\qquad+\frac{1}{(n+1+x)^2}\left[-\frac{x(1-v)^3}{2}\frac{n+1}{n+1+x}+\frac{x^2(1-v)^2}{2}\right.\\
  &\qquad\left.+\frac{(1-v)^3}{3}\frac{n+1}{n+1+x}+\frac{(1-v)^4}{8}\left(\frac{n+1}{n+1+x}\right)^2\right]+o\left(\frac{1}{n^2}\right).
  \end{align*}
  Let $\rho_n(x)=\frac{n+1}{n+1+x}=1-\frac{x}{n+1+x}$, then
  \begin{align*}
  c_n(v,x)&=1+\frac{1}{n+1+x}\left((1-v)x-\frac{(1-v)^2}{2}\rho_n(x)\right)\\
  &\qquad+\frac{1}{(n+1+x)^2}\left[-\frac{x(1-v)^3}{2}\rho_n(x)+\frac{x^2(1-v)^2}{2}+\frac{(1-v)^3}{3}\rho_n(x)\right.\\
  &\qquad\left.+\frac{(1-v)^4}{8}\rho_n(x)^2\right]+o\left(\frac{1}{n^2}\right)\\
  &=1+\frac{1}{n+1+x}\left(x(1-v)-\frac{(1-v)^2}{2}+\frac{x(1-v)^2}{2(n+1+x)}\right)\\
  &\qquad+\frac{1}{(n+1+x)^2}\left(-\frac{x(1-v)^3}{2}+\frac{x^2(1-v)^2}{2}-\frac{(1-v)^3}{3}\right.\\
  &\qquad\left.+\frac{(1-v)^4}{8}\right)+o\left(\frac{1}{n^2}\right)\\
  &=1+\frac{1}{n+1+x}\left(x(1-v)-\frac{(1-v)^2}{2}\right)\\
  &\qquad+\frac{1}{(n+1+x)^2}\left(\frac{(1-v)^2}{2}(x^2+x)-(1-v)^3\left(\frac{x}{2}+\frac{1}{3}\right)\right.\\
  &\qquad\left.+\frac{(1-v)^4}{8}\right)+o\left(\frac{1}{n^2}\right).
  \end{align*}

  (ii) It can be easily deduced from Equation \eqref{eq:cnv}. 
  
  (iii) Starting from Equation \eqref{eq:cnv}, since $\int_0^1 (1-v)^kdv =\frac{1}{k+1}$,we have
  \begin{align*}
  \int_0^1 c_n(v,x)dv&=\int_0^1 \left[1+\frac{1}{n+1+x}\left(x(1-v)-\frac{(1-v)^2}{2}\right)\right.\\
  &+\left.\frac{1}{(n+1+x)^2}\left(\frac{(1-v)^2}{2}(x^2+x)-(1-v)^3\left(\frac{x}{2}+\frac{1}{3}\right)\right.\right.\\
  &\qquad\left.\left.+\frac{(1-v)^4}{8}\right)\right]dv+o\left(\frac{1}{n^2}\right)\\
  &=1+\frac{1}{n+1+x}\left(\frac{x}{2}-\frac{1}{6}\right)\\
  &\qquad+\frac{1}{(n+1+x)^2}\left(\frac{x^2}{6}+\frac{x}{6}-\frac{x}{8}-\frac{1}{12}+\frac{1}{40}\right)+o\left(\frac{1}{n^2}\right)\\
  &=1+\frac{1}{n+1+x}\left(\frac{x}{2}-\frac{1}{6}\right))\\
  &\qquad+\frac{1}{(n+1+x)^2}\left(\frac{x^2}{6}+\frac{x}{24}-\frac{7}{120}\right)+o\left(\frac{1}{n^2}\right).
  \end{align*}

  (iv) Using Taylor expansions, we get
  \[
  \rho'_n(x)=\frac{n+1}{n+x}=1+\frac{1-x}{n+1+x}+\frac{1-x}{(n+1+x)^2}+o\left(\frac{1}{n^2}\right). 
  \]
  We expand $c_n(0,x)\rho'_n(x)$ using Equation \eqref{eq:intcnv}.
  \begin{align*}
  c_n(0,x)\rho'_n(x)&=\left(1+\left(\frac{1}{n+1+x}\right)\left(x-\frac{1}{2}\right)+\frac{1}{(n+1+x)^2}\left(\frac{x^2}{2}-\frac{5}{24}\right)\right)\\
  &\qquad\times \left(1+\frac{1-x}{n+1+x}+\frac{1-x}{(n+1+x)^2}\right)+o\left(\frac{1}{n^2}\right)\\
  &=1+\frac{1}{n+1+x}\left[x-\frac{1}{2}+1-x\right]+\frac{1}{(n+1+x)^2}\\
  &\qquad\times\left[\frac{x^2}{2}-\frac{5}{24}+(1-x)+\left(x-\frac{1}{2}\right)(1-x)\right]+o\left(\frac{1}{n^2}\right)\\
  &=1+\frac{1/2}{n+1+x}+\frac{1}{(n+1+x)^2}\left[\frac{x^2}{2}-\frac{5}{24}+1-x-x^2-\frac{1}{2}+\frac{x}{2}\right]\\
  &\qquad+o\left(\frac{1}{n^2}\right)\\
  &=1+\frac{1/2}{n+1+x}+\frac{1}{(n+1+x)^2}\left[-\frac{x^2}{2}+\frac{x}{2}+\frac{7}{24}\right]+o\left(\frac{1}{n^2}\right)
  \end{align*}
  whereby the result is deduced.
  
  (v) Let $x+r_n(x,k)\in[-1/3,2/3)$, then 
  \begin{align*}
  &r_{n+1}(x,k)-c_n(0,x)\rho'_n(x)r_n(x,k)\\
  &=r_{n+1}(x,k)-r_n(x,k)+r_n(x,k)(1-c_n(0,x)\rho'_n(x))\\
  &=\frac{k}{n+1+x}-\frac{k}{n+x}+\frac{k}{n+x}\left(\frac{-1/2}{n+1+x}\right)+o\left(\frac{1}{n^2}\right)\\
  &=\frac{k(n+x)-k(n+1+x)}{(n+x)(n+1+x)}+\frac{k}{n+x}\left(\frac{-1/2}{n+1+x}\right)+o\left(\frac{1}{n^2}\right)\\
  &=\frac{-k-k/2}{(n+x)(n+1+x)}+o\left(\frac{1}{n^2}\right)\\
  &=\frac{-3k/2}{(n+x)(n+1+x)}+o\left(\frac{1}{n^2}\right)\\
  &=\frac{-3k/2}{(n+1+x)^2}\frac{n+1+x}{n+x}+o\left(\frac{1}{n^2}\right)\\
  &=\frac{-3k}{2}\frac{1}{(n+1+x)^2}+o\left(\frac{1}{n^2}\right).
  \end{align*}
  
  (vi)  We use the Taylor expansion of $\frac{n+1+x}{n+2}$, which is $1+\frac{x-1}{n+1+x}+\frac{(x-1)^2}{(n+1+x)^2}+o\left(\frac{1}{n^2}\right)$ to deduce that
  \begin{align*}
  &\left(\frac{n+1+x}{n+2}\right)-c_n(0,x)\\
  &=1+\frac{x-1}{n+1+x}+\frac{(x-1)^2}{(n+1+x)^2}-1-\frac{x-1/2}{n+1+x}-\frac{x^2/2-5/24}{(n+1+x)^2}+o\left(\frac{1}{n^2}\right)\\
  &=\frac{-1/2}{n+1+x}-\frac{(x-1)^2-x^2/2+5/24}{(n+1+x)^2}+o\left(\frac{1}{n^2}\right)\\
  &=\frac{-1/2}{n+1+x}-\frac{1}{(n+1+x)^2}\left(x^2-2x+1-\frac{x^2}{2}+\frac{5}{24}\right)+o\left(\frac{1}{n^2}\right)\\
  &=\frac{-1/2}{n+1+x}-\frac{1}{(n+1+x)^2}\left(\frac{x^2}{2}-2x+\frac{29}{24}\right)+o\left(\frac{1}{n^2}\right).\\
  \end{align*}
  
  (vii) Now, let $x+r_n(x,k)\in[2/3,5/3)$. We use the Taylor expansion of $\frac{n+1+x}{n+2}$, which is $1+\frac{x-1}{n+1+x}+o\left(\frac{1}{n^2}\right)$ to deduce that
  \begin{align*}
  &r_{n+1}(x,k)\frac{n+1+x}{n+2}-r_n(x,k)c_n(0,x)=\frac{k}{n+1+x}\frac{n+1+x}{n+2}-\frac{k}{n+x}c_n(0,x)\\
  &=\frac{k}{n+1+x}\left(1+\frac{x-1}{n+1+x}\right)-\frac{k}{n+x}\left(1+\frac{x-1/2}{n+1+x}\right)+o\left(\frac{1}{n^2}\right)\\
  &=\frac{k}{n+1+x}+\frac{k(x-1)}{n+1+x}-\frac{k}{n+x}-\frac{k(x-1/2)}{n+x}+o\left(\frac{1}{n^2}\right)\\
  &=\frac{-k}{(n+x)(n+1+x)}+\frac{k(x-1)(n+x)-k(x-1/2)(n+1+x)}{(n+x)(n+1+x)^2)}+o\left(\frac{1}{n^2}\right)\\
  &=\frac{-k(n+1+x)+k(x-1)(n+x)-k(x-1/2)(n+1+x)}{(n+x)(n+1+x)^2}+o\left(\frac{1}{n^2}\right)\\
  &=\frac{(n+x)\left[-k\frac{n+1+x}{n+x}+k(x-1)-k(x-1/2)\frac{n+1+x}{n+x}\right]}{(n+x)(n+1+x)^2)}+o\left(\frac{1}{n^2}\right)\\
  &=\frac{-k\frac{n+1+x}{n+x}+k(x-1)-k(x-1/2)\frac{n+1+x}{n+x}}{(n+1+x)^2)}+o\left(\frac{1}{n^2}\right)\\
  &=\frac{-k+k(x-1)-k(x-1/2)}{(n+1+x)^2)}+o\left(\frac{1}{n^2}\right)\\
  &=\frac{-3k}{2}\frac{1}{(n+1+x)^2}+o\left(\frac{1}{n^2}\right).
  \end{align*}
\end{proof}
\end{document}